\documentclass[11 pt, twoside]{amsart}
\usepackage{parskip}
\usepackage{amsthm,amsmath,amssymb,oldgerm}
\usepackage[a4paper, margin=2.5cm]{geometry}
\usepackage{mathrsfs}
\usepackage{verbatim}
\usepackage{hyperref}
\theoremstyle{plain}
\newtheorem{thm}{Theorem}[section]
\newtheorem*{theorem*}{Main theorem}

\newtheorem{lem}[thm]{Lemma}

\newtheorem{cor}[thm]{Corollary}
\theoremstyle{definition}
\theoremstyle{definition}
\newtheorem{rem}[thm]{Remark}

\def\Bn{\mathbb{B}^{n}}
\def\Btwo{\mathbb{B}^{2}}

\def \G {\Gamma}

\def \C {\mathbb{C}}

\def \R {\mathbb{R}}

\def \calf {\mathcal{F}}

\def \cala {\mathcal{A}}

\def \calo {\mathcal{O}}
\def \cals {\mathcal{S}}

\def\SU{\mathrm{SU}}

\DeclareMathOperator{\bk} {{\it{\mathcal{B}_{\Gamma}^{k}}}}
\DeclareMathOperator{\bkm} {{\it{\mathcal{B}_{M}^{\lambda\otimes\ell^{\otimes k}}}}}

\DeclareMathOperator{\ck}{{{\it{C}}_{\Gamma,{\it{k}}}}}
\DeclareMathOperator{\dhyp}{d_{\mathrm{hyp}}}

\DeclareMathOperator{\shyp}{\mu_{shyp}^{vol}}
\DeclareMathOperator{\hypbn}{\mu_{hyp}}
\DeclareMathOperator{\hypbnvol}{\mu_{hyp}^{vol}}

\DeclareMathOperator{\vx}{\mathrm{vol_{\mathrm{hyp}}}}

\DeclareMathOperator{\lk}{{\it{\overline{\square}_{k}}}}
\DeclareMathOperator{\hkm} {{\it{K_{M,\lambda\otimes\ell^{\otimes k}}}}}
\DeclareMathOperator{\rx}{{\it{r_{X_{\G}}}}}
\title[Estimates of Picard modular cusp forms]{Estimates of Picard modular cusp forms}
\author{Anilatmaja Aryasomayajula}
\address{Department of Mathematics, Indian Institute of Science Education and Research (IISER) Tirupati, 
Transit campus at Sri Rama Engineering College, Karkambadi Road,
Mangalam (B.O),Tirupati-517507, India.}
\email{anil.arya@iisertirupati.ac.in}
\author{Baskar Balasubramanyam}
\address{Department of Mathematics, Indian Institute of Science Education and Research (IISER) Pune,
Dr. Homi Bhabha Road, Pashan, Pune 411008, India.}
\email{baskar@iiserpune.ac.in}
\author{Dyuti Roy}
\address{Department of Mathematics, Indian Institute of Science Education and Research (IISER) Tirupati, 
Transit campus at Sri Rama Engineering College, Karkambadi Road,
Mangalam (B.O),Tirupati-517507, India.}
\email{dyutiroy@students.iisertirupati.ac.in}
\date{\today}
\begin{document}
\begin{abstract}
In this article, for $n\geq 2$, we compute asymptotic, qualitative, and quantitative estimates of the Bergman kernel of Picard modular cusp forms associated to torsion-free, cocompact subgroups of $\mathrm{SU}\big((n,1),\C\big)$. The main result of the article is the following result. Let $\G\subset \mathrm{SU}\big((2,1),\calo_{K}\big)$ be a torsion-free subgroup of finite index, where $K$ is a totally imaginary field. Let $\bk$ denote the Bergman kernel associated to the $\cals_{k}(\Gamma)$, complex vector space of weight-$k$ cusp forms with respect to $\G$. Let $\Btwo$ denote the $2$-dimensional complex ball endowed with the hyperbolic metric, and let $X_{\G}:=\G\backslash \Btwo$ denote the quotient space, which is a noncompact complex manifold of dimension $2$. Let $\big|\cdot\big|_{\mathrm{pet}}$ denote the point-wise Petersson norm on $\cals_{k}(\Gamma)$. Then, for $k\geq 6$, we have the following estimate
\begin{equation*}
\sup_{z\in X_{\G}}\big|\bk(z)\big|_{\mathrm{pet}}=O_{\G}\big(k^{\frac{5}{2}}\big),
\end{equation*}
where the implied constant depends only on $\G$.
\end{abstract}
\maketitle

\vspace{0.2cm}
\section{Introduction}
\subsection{History and background} 
Estimates of automorphic forms is a subject of great interest in recent times, owing to their applications both in arithmetic intersection theory, and arithmetic quantum chaos. Especially, estimates of the Bergman kernel associated to the vector bundle of automorphic forms is of arithmetic and geometric significance.   

\vspace{0.1cm}
Estimates of Bergman kernel associated to cusp forms defined over hyperbolic Riemann surfaces were derived in \cite{abbes} and \cite{jk}. Furthermore, off-diagonal estimates of these Bergman kernels were derived in \cite{anil1} and \cite{anil2}, which were useful  in deriving sub-convexity estimates of Hecke eigen cusp forms associated to certain cocompact Fuchsian subgroups arising from quaternion algebras in \cite{ab2}. 

\vspace{0.1cm}
In \cite{antareep}, optimal estimates of the Bergman kernel associated to Siegel modular cusp forms were derived in both the cocompact and cofinite setting. 

\vspace{0.1cm}
In this article, we extend methods from \cite{anil1}, \cite{anil2}, \cite{anil3}, and \cite{ab1}, to derive estimates of Picard modular cusp forms. For $n\geq 2$, we prove the asymptotic, qualitative, and quantitative estimates of the Bergman kernel associated to Picard modular cusp forms defined with respect to cocompact subgroups of $\SU\big((n,1),\C\big)$. 

\vspace{0.1cm}
Furthermore, we derive quantitative estimates of the Bergman kernel associated to Picard modular cusp forms defined with respect to finite index subgroups of  $\SU\big((2,1),\calo_{K}\big)$, where $\calo_{K}$ is the ring of integers of a totally imaginary quadratic field.    

\vspace{0.2cm}
\paragraph{\bf{Notation.}} For $n\geq 2$, let 
\begin{align*}
\Bn:=\big\lbrace z:=\big(z_{1},\dots,z_{n}\big)^{t}\in\C^{n}\big|\,\big|z\big|^{2}:=\big|z_{1}\big|^{2}+\cdots+\big|z_{n}\big|^{2}<1\big\rbrace
\end{align*}
denote the complex hyperbolic ball endowed with the hyperbolic metric $\hypbn$, and let $\Gamma \subset \SU\big((n,1),\C\big)$ be a discrete, torsion-free, cofinite subgroup, acting on $\Bn$ via fractional linear transformations. Let $X_{\G}=\Gamma \backslash \Bn$ denote the quotient space, which is an $n$-dimensional complex manifold of finite hyperbolic volume. 

\vspace{0.1cm}
For $k\geq 1$, let $\cals_{k}(\Gamma)$ denote the complex vector space of weight-$k$ cusp forms, and let $\lbrace f_{1},\ldots f_{d_{k}} \rbrace$ denote an orthonormal basis of $\cals_{k}(\Gamma)$ with respect to the Petersson inner-product. 

\vspace{0.1cm}
With notation as above, for $n\geq 2$, and $k\geq 1$, and $z,w\in\Bn$, the Bergman kernel is given by the following formula
\begin{align*}
\bk(z,w):=\sum_{j=1}^{d_{k}}f_{j}(z)\overline{f_{j}(w)},
\end{align*}
which is a holomorphic cusp form in the $z$-variable, and an anti-holomorphic cusp form in the $w$-variable. The definition of the Bergman kernel is independent of the choice of the orthonormal basis of $\cals_{k}(\Gamma)$. 

\vspace{0.1cm}
The Petersson norm of the Bergman kernel at any $z \in \Bn$ is given by the following formula
\begin{align*}
\big|\bk(z)\big|_{\mathrm{pet}}:= \big( 1- \big|z\big|^2\big)^{k}\; \big|\bk(z,z)\big|= \big( 1- \big|z\big|^2\big)^{k}\sum_{j=1}^{k}\big|f_{j}(z)\big|^{2}.
\end{align*}

\vspace{0.2cm} 
\subsection{Statement of main results}
We now state the first main result of the article, which is proved as Theorem \ref{mainthm1}.

\vspace{0.1cm}
\textbf{Theorem 1.}
With notation as above, for $n\geq 2$, let $\G\subset  \SU\big((n,1),\C\big)$ be a discrete, torsion-free, cocompact subgroup. Then, for $z\in X_{\G}$,  we have 
\begin{align*}
\lim_{k\rightarrow \infty}\frac{1}{k^n}\big|\bk(z)\big|_{\mathrm{pet}}=  \frac{1}{(4\pi)^{n}},
\end{align*}
and the convergence of the above limit is uniform in $z\in X$.

\vspace{0.1cm}
The second main result of the article is the following theorem, which is proved as Corollary \ref{cor3}.

\vspace{0.1cm}
\textbf{Theorem 2.}
With notation as above, for $n\geq 2$, let $\G\subset  \SU\big((n,1),\C\big)$ be a discrete, torsion-free, cocompact subgroup. Then, for $k\geq 2n+2$, we have the following estimate
\begin{align}\label{eq:theorem2}
\sup_{z\in X_{\G}}\big|\bk(z)\big|_{\mathrm{pet}}=O_{\G}\big(k^{n}\big),
\end{align}
where the implied constant depends only on $\G$. 

\vspace{0.1cm}
The third result of the article is the following theorem, which is proved as Corollary \ref{cor4}.

\vspace{0.1cm}
\textbf{Theorem 3.}
With notation as above, let $\Gamma \subset\mathrm{SU}((2,1),\calo_K) $ be a torsion-free subgroup of finite index, where $K$ is a totally imaginary quadratic field. Furthermore, we assume that $\G$ admits only one cusp at $\infty$. Then, for $k\geq 6$, we have the following estimate
\begin{align}\label{eq:theorem3}
\sup_{z\in X_{\G}}\big|\bk(z)\big|_{\mathrm{pet}}=O_{\G}\big(k^{5\slash 2}\big), 
\end{align}
where the implied constant depends only on $\G$. 

\vspace{0.1cm}
\begin{rem}
The implied constants in estimates \eqref{eq:theorem2} and \eqref{eq:theorem3}, can be explicitly computed, and also remain stable in covers, which we demonstrate in Remark \ref{rem:covers}.

\vspace{0.1cm}
The assumption that $\G$ admits only one cusp at $\infty$ in Theorem 3, is only for notational convenience, and estimate \eqref{eq:theorem3} easily extends to cofinite subgroups with multiple cusps. Furthermore, torsion points should not be much of an impediment for extending estimate \eqref{eq:theorem3} to subgroups with torsion points.  

\vspace{0.1cm}
Lastly, estimate \eqref{eq:theorem3} can be easily extended to cofinite subgroups, which are comensurable with $\mathrm{SU}((2,1),\calo_K)$. 
\end{rem}

\vspace{0.1cm}
\begin{rem}
Let $\G\subset \mathrm{SU}\big((n,1),\C\big)$  be a discrete, cofinite subgroup. Following results from \cite{jk}, \cite{ab1} and \cite{antareep}, the following estimate  is expected, which has now assumed the status of a folk-lore conjecture
\begin{align*}
\sup_{z\in X_{\G}}\big|\bk(z)\big|_{\mathrm{pet}}=O_{\G}\big(k^{3n\slash 2}\big).
\end{align*}
However, estimate \eqref{eq:theorem3} suggests that we can expect a finer estimate than the above estimate for semi-simple Lie groups acting on symmetric spaces of rank $1$. 
\end{rem}

\vspace{0.2cm}
\section{Background material}\label{Background material}
We now set up the notation, which we use for the rest of the article and recall the background material for the proofs of the main results. 

\vspace{0.25cm} 
\subsection{The ball model}\label{ballmodel}
For $n\geq 2$, let  $H$  be a Hermitian matrix of signature $(n,1)$, and let $ \big(\mathbb{C}^{n+1},H\big)$ denote the Hermitian inner-product space $\C^{n+1}$ equipped with the Hermitian inner-product $\langle\cdot, \cdot\rangle_{H}$, induced by $H$. For $z,w \in \mathbb{C}^{n+1}$, the Hermitian inner-product $\langle\cdot, \cdot\rangle_{H}$ is defined as 
\begin{align*}
\langle z, w \rangle_H = w^{*}H z,
\end{align*}
where $w^{*}$ denotes the complex transpose of $w$.

\vspace{0.1cm}
Set
\begin{align*}
V_{-}(H):= \big\lbrace z\in\mathbb{C}^{n+1}\big|\,\langle z, z \rangle_H < 0 \big\rbrace\subset \C^{n+1};\,\,
 V_0(H):= \big\lbrace z\in\mathbb{C}^{n+1}\big|\,\langle z, z \rangle_H = 0 \big\rbrace\subset \C^{n+1}.
\end{align*}

Set 
\begin{align*}
\cala^{n+1}:=\big\lbrace \big(z_{1},\ldots,z_{n+1}\big)^{t}\in \C^{n+1}\big|\,z_{n+1}\not=0\rbrace\subset \C^{n+1}.
\end{align*}

The subspace $\cala^{n+1}$ inherits the structure of a Hermitian inner-product space, which we denote by $\big(\cala^{n+1},H\big)$. We have the following map
\begin{align*}
&\phi:\cala_{n+1} \longrightarrow \C^{n}\\&
\big(z_1,\ldots,z_{n+1}\big)^{t}\mapsto \bigg( \frac{z_1}{z_{n+1}},  \cdots ,\frac{z_n}{z_{n+1}}\bigg)^{t}.  
\end{align*}                    

The complex hyperbolic space  $\mathcal{H}^{n}\left(\mathbb{C}\right)$  and its boundary $\partial \mathcal{H}^{n}\left(\mathbb{C}\right)$ are defined as 
\begin{align*}
\mathcal{H}^{n}\left(\mathbb{C}\right):=\phi\big( V_{-}(H)\big)\subset \C^{n};\,\,\,\partial \mathcal{H}^{n}\left(\mathbb{C}\right):=\phi\big( V_{0}(H)\big)\subset\C^{n}.
\end{align*}

\vspace{0.1cm}
The choice of Hermitian matrix gives different models of hyperbolic $n$-space. For
\begin{align}\label{defn:H}
H = \begin{pmatrix}
\mathrm{Id}_{n } & 0 \\
0  & -1
\end{pmatrix},
\end{align}
we obtain the unit ball model of the complex hyperbolic space, where $\mathrm{Id}_{n}$ denotes the identity matrix of order $n$, which we now describe.

\vspace{0.1cm}
For any $z= \big(z_1,\ldots ,z_n\big)^{t} \in \mathcal{H}^n(\C)$, consider the lift $\tilde{z}:=\big(z_1,\ldots ,z_n,1\big)^{t}\in\cala^{n+1}$. Observe that
\begin{align*}
\langle \tilde{z},\tilde{z}\rangle_{H} < 0
\iff  |z_1|^2 +\cdots +|z_n|^2 <1.
\end{align*}

Thus $\mathcal{H}^n(\C)$ can be identified with the open unit ball 
\begin{align*}
\mathbb{B}^n := \big\lbrace z:=\big(z_{1},\ldots,z_{n}\big)^{t}\in \C^{n}\big|\, \big|z\big|^{2}:=\big|z_1\big|^2+ \cdots + \big|z_n\big|^2 < 1 \big\rbrace.
\end{align*}

This model which is also known as the ball model, is the natural generalization of Poincare disk model of upper half-plane $\mathbb{H}$.
 
\vspace{0.1cm} 
The space is endowed with its usual K\"ahler-Bergman metric $\mu_{\mathrm{hyp}}$,  which has constant negative curvature $-1$. For any given $z=(z_{1},\ldots,z_{n})\in \mathbb{B}^n$, the hyperbolic metric $\mu_{\mathrm{hyp}}(z)$  is given by the following formula
\begin{align}\label{defnhypmetric}
\mu_{\mathrm{hyp}}(z):=-2i\partial\overline{\partial}\log(1-|z|^{2}),
\end{align}
and let $\mu_{\mathrm{hyp}}^{\mathrm{vol}}(z)$ denote the associated volume form. 

\vspace{0.1cm}
From standard results in hyperbolic geometry, for any $z=(z_1,\ldots,z_n)^{t},\,w=(w_1,\ldots,w_n)^{t}\in \mathbb{B}^n$ with respective lifts $\tilde{z}=(z_1,\ldots,z_n,1)^{t},\,\tilde{w}=(w_1,\ldots,w_n,1)^{t}\in \cala^{n+1}$, the hyperbolic distance in the ball model is given by the following relation
\begin{align}\label{coshreln}
&\cosh^{2}\big(\dhyp(z,w)\slash 2\big)=\frac{\langle \tilde{z},\tilde{w}\rangle_H\; \langle \tilde{w},\tilde{z} \rangle_H }{\langle \tilde{z},\tilde{z}\rangle_H\; \langle \tilde{w},\tilde{w}\rangle_H},\,\,\,\;\mathrm{where}\,\,\langle \tilde{z},\tilde{w} \rangle _H:=\tilde{w}^{\ast}H\tilde{z}=\sum_{j=1}^{n}z_j\overline{w}_{j}-1,
\end{align}
and $H$ is as described in equation \eqref{defn:H}.

\vspace{0.1cm} 
The unitary group $\mathrm{SU}\big((n,1),\C\big)$ is given by 
\begin{align*}
\mathrm{SU}\big((n,1),\C\big):= \{ A \in \mathrm{SL}\big(n,\C\big) \; | \, A^*H A =  H\},\, \,\mathrm{where} \,\,
H = \begin{pmatrix}
\mathrm{Id}_{n } & 0 \\
0  & -1
\end{pmatrix},
\end{align*}
and $\mathrm{Id}_{n}$ denotes the identity matrix of order $n$.

\vspace{0.1cm} 
The group $ \mathrm{SU}\big((n,1),\C\big)$ acts on $\mathbb{B}^n$ via bi-holomorphic mappings, which are given by fractional linear transformations, which we explain here. Let 
\begin{align*}
\gamma=\left(\begin{array}{cc} A&B\\C &D\end{array}\right)\in \mathrm{SU}\big((n,1),\C\big),
\end{align*}
where $A\in M_{n\times n}(\C)$, $B\in M_{n\times 1}(\C)$, $C\in M_{1\times n}(\C) $, and $D\in\C$. For $z=(z_1,\ldots,z_n)^{t}\in\Bn$, the action of $\gamma$ on $\mathbb{B}^n$, is given by 
\begin{align}\label{gammact}
\gamma z:=\frac{Az+B}{Cz+D}.
\end{align}
Here, $z$ is seen as a column vector vector in $\C^n$. 

\vspace{0.1cm} 
Let $\Gamma\subset \mathrm{SU}\big((n,1),\C\big)$ be a discrete, torision-free, cofinite subgroup, and let $X_{\G}:=\Gamma\backslash\mathbb{B}^n$ denote the quotient space, which  is called a Picard modular variety of dimension $n$. The Picard modular variety $X_{\G}$ admits the structure of a complex manifold of dimension $n$.  The hyperbolic metric $\hypbn(z)$ is $\mathrm{SU}\big((n,1),\C\big)$-invariant, and hence, defines a K\"ahler metric on $X_{\G}$.

\vspace{0.1cm} 
When $\G$ is a cocompact subgroup, the Picard modular variety $X$ is a compact complex manifold of dimension $n$. When $\G$ is cofinite, then the Picard modular variety $X_{\G}$  is a non-compact complex manifold of dimension $n$, of finite hyperbolic volume. 

\vspace{0.25cm} 
\subsection{Three models of $\mathcal{H}^{2}(\C)$ }\label{noncompactmodel}
For $n=2$, we use  different models of the hyperbolic space, which we now describe in detail.

\textbf{Model (1)}: As in Section \eqref{ballmodel}, the ball model 
\begin{align*}
\mathbb{B}^2:=\big\lbrace z=(z_1,z_2)^{t} \in \mathbb C^2 \big|\,\big|z\big|^2:= \big|z_1 \big| ^2 + \big|z_2 \big|^2 < 1 \big\rbrace
 \end{align*}
is obtained by setting 
\begin{align*}
H={\begin{pmatrix} 1 & 0 & 0\\ 0 & 1 & 0 \\ 0& 0 & -1   \end{pmatrix}}.
\end{align*}
         
\textbf{{Model (2)}}: For $z=(z_1,z_2,z_3)^t$ and $w=(w_1,w_2,w_3)^t \in \mathbb{C}^{3}$, consider the Hermitian form
\begin{align} \label{innerpro2}
\langle z, w \rangle_2 = iz_1 \overline{w_3}+ z_2 \overline{w_2}-i z_3 \overline{w_1}, 
\end{align}

which is determined by the following Hermitian matrix of signature $(2,1)$
\begin{align*}
H_2={\begin{pmatrix} 0 & 0 & -i\\ 0 & 1 & 0 \\ i & 0 & 0   \end{pmatrix}}.
\end{align*}

\vspace{0.1cm}
For any $z= \big(z_1,z_2\big)^{t}\in \mathcal{H}^2(\C)$, with the lift $\tilde{z}:=\big(z_1,z_2,1\big)^{t}\in\cala^{3}$, observe that
\begin{align*}
\langle \tilde{z},\tilde{z}\rangle_{2} < 0\iff 2\mathrm{Im}(z_1) - \big|z_{2} \big|^2>0.
\end{align*}

Thus $\mathcal{H}^2(\C)$ can be identified with the following space 
\begin{align*}
\mathbb{B}^2_2:=\big\lbrace z=(z_1,z_2)^{t} \in \mathbb \C^2 \big|\, 2\mathrm{Im}(z_1) - \big| z_{2} \big|^2  > 0 \big\rbrace.
\end{align*}

\textbf{{Model (3)}}:  For $z=(z_1,z_2,z_3)^t$ and $w=(w_1,w_2,w_3)^t \in \mathbb{C}^{3}$, consider the Hermitian form
\begin{align} \label{innerpro3}
\langle z, w \rangle_3 = z_1 \overline{w_3}+ z_2 \overline{w_2} + z_3 \overline{w_1}, 
\end{align}

which is determined by the following Hermitian matrix of signature $(2,1)$
\begin{align*}
H_3={\begin{pmatrix} 0 & 0 & 1\\ 0 & 1 & 0 \\ 1 & 0 & 0   \end{pmatrix}}.
\end{align*}

\vspace{0.1cm}
For any $z= \big(z_1,z_2\big)^{t}\in \mathcal{H}^2(\C)$ with the lift $\tilde{z}:=\big(z_1,z_2,1\big)^{t}\in\cala^{3}$, observe that
\begin{align*}
\langle \tilde{z},\tilde{z}\rangle_{3} < 0\iff 2 \mathrm{Re}(z_1) + \big|z_{2} \big| ^2  < 0.
\end{align*}

Thus $\mathcal{H}^2(\C)$ can be identified with the following space  
\begin{align*}
\mathbb{B}^2_3:=\big\lbrace z=(z_1,z_2)^{t} \in  \C^2 \big|\, 2\mathrm{Re}(z_1) + \big| z_{2}\big|^2  < 0 \big\rbrace.
\end{align*}
 
These are the three models which we work with, to prove estimate \eqref{eq:theorem3}. We treat the ball model as the base model, and we go back and forth between the models as per our convenience.

\vspace{0.1cm}
The Cayley transformation
\begin{align}\label{cayley13}
\gamma_{3} := {\begin{pmatrix} 1 & 1 & 0 \\ 0 & 1 & -1 \\ 1 & 1 & -1 \end{pmatrix}}:\big( \cala^{3},H\big)\longrightarrow :\big( \cala^{3},H_3\big)
\end{align}
defines an isometry. 
 
\vspace{0.1cm}  
The Cayley transformation
\begin{align}\label{cayley23}
\gamma_{23}: = {\begin{pmatrix} i & 0 & 0 \\ 0 & 1 & 0 \\ 0 & 0 & 1 \end{pmatrix}}:\big( \cala^{3},H_2\big)\longrightarrow \big( \cala^{3},H_3\big)
\end{align}
defines an isometry.

\vspace{0.1cm}
From the above Cayley transformations, it is easy to see that the map 
\begin{align}\label{cayley}
\gamma_{2}:= \gamma_{23}^{-1}\circ \gamma_3:\big( \cala^{3},H\big)\longrightarrow \big( \cala^{3},H_2\big)
\end{align}
defines an isometry.

\vspace{0.1cm}
Let $\G\subset \mathrm{SU}\big((2,1),\mathcal{O}_{K}\big)$ be a torsion-free, finite index subgroup, with only one cusp at $\infty$. As discussed in Section \eqref{ballmodel}, the quotient space $X_{\G}:=\G\backslash \mathbb{B}^2$ admits the structure of a noncompact complex manifold of dimension $2$, of finite hyperbolic volume. 

\vspace{0.1cm}
The Picard modular variety $X_{\G}$ is isometric to $\G^2\backslash \mathbb{B}^2_2$ and $\G^3\backslash \mathbb{B}^2_3$, where $\G_2:=\gamma_2\G\gamma_2^{-1}$ and 
$\G_3:=\gamma_3\G\gamma_3^{-1}$.

\vspace{0.1cm}
Let $\Gamma_{\infty}$ denote the stabilizer of the cusp $\infty.$ Let $\Gamma^2_{\infty}$ and $\Gamma^3_{\infty}$ be the corresponding stabilizer group in the domains $\mathbb{B}^2_2$ and $\mathbb{B}^2_3$, respectively.  Furthermore, we have following relation
\begin{align}\label{gammainfty}
\Gamma^3_{\infty} = \gamma_{3} \Gamma^1_{\infty}\gamma_{3}^{-1}= \gamma_{23} \Gamma^2_{\infty}\gamma_{23}^{-1};
\end{align}
where $\gamma_{13}$ and $\gamma_{23}$ are described in the equations \eqref{cayley13} and \eqref{cayley23}, respectively.

\vspace{0.1cm}
As discussed in \cite{holzapfel}, the stabilizer subgroup of $\infty$ in the group $\gamma_2\G_{0}\gamma_2^{-1}$, is explicitly described as
\begin{align}\label{B2stabilizer}
\Gamma^2_{0,\infty} = \left\{ \begin{pmatrix} 
1 & i\overline{\alpha} & i \frac{|\alpha|^2}{2} +\beta \\0 & 1& \alpha \\0 & 0 & 1              
\end{pmatrix} \middle| \hspace{.2cm} \alpha\in \mathbb{C}, \beta \in \mathbb{R} \right\},
\end{align}
where $\G_{0}:=\mathrm{SU}\big((2,1),\calo_{K}\big)$, where $K$ is a totally imaginary quadratic field. 
 
Combining equations \eqref{gammainfty} and \eqref{B2stabilizer}, the stabilizer subgroup of $\infty$ in the group $\gamma_3\G_{0}\gamma_3^{-1}$, is explicitly described as
\begin{align}\label{B3stabilizer}
\Gamma^3_{0,\infty} = \left\{ \begin{pmatrix} 
1 & -\overline{\alpha} & - \frac{|\alpha|^2}{2} + i\beta \\0 &1 & \alpha\\ 0 & 0 & 1              
\end{pmatrix} \middle| \hspace{.2cm} \alpha \in \mathbb{C}, \beta \in \mathbb{R} \right\}.
 \end{align}

\vspace{0.2cm}
\subsection{Picard modular cusp forms and the Bergman kernel}\label{picardcuspforms}
With hypothesis as above, for $n\geq 2$, let $\G\subset \mathrm{SU}\big((n,1),\C\big)$ be a discrete, torsion-free, cofinite subgroup, and let $X_{\G}=\G\backslash \mathbb{B}^n$ denote the quotient space. For any $z,w\in X_{\G}$, let $\dhyp(z,w)$  denote the geodesic distance between the points $z$ and $w$ on the Picard modular variety $X_{\G}$,  i.e., $\dhyp$ is the natural distance function associated to the hyperbolic metric $\hypbn$ on $X_{\G}$.  

\vspace{0.1cm}
Locally, we identify $X_{\G}$ with its universal cover $\mathbb{B}^{n}$. For $z,w\in X_{\G}$, the hyperbolic distance function $\dhyp(z,w)$ satisfies equation \eqref{coshreln}.

\vspace{0.1cm}
Let $\vx(X_{\G})$ denote the volume of $X_{\G}$ with respect to the hyperbolic volume form  $\hypbnvol(z)$, and let 
\begin{align*}
\shyp:=\frac{\hypbnvol}{\vx\big(X_{\G}\big)},
\end{align*}
denote the rescaled hyperbolic volume form, which measures the volume of $X_{\G}$ to be one. 

\vspace{0.1cm}
When $\G$ is cocompact, the injectivity radius of $X_{\G}$ is given by the following formula 
\begin{align}\label{ircom}
\rx := \inf \big\lbrace \dhyp(z,\gamma z)|\,z \in X_{\G}, \gamma\in\Gamma\backslash\mathrm{Id}\big\rbrace, 
\end{align}
where $\mathrm{Id}$ is the identity matrix.

\vspace{0.1cm}
When $\G\subset \mathrm{SU}\big((2,1),\calo_{K}\big)$ is a torsion-free, finite index subgroup with only one cusp at $\infty$, the injectivity radius of $X_{\G}$ is given by the following formula
\begin{align}\label{irnoncomp}
\rx := \inf \big\lbrace \dhyp(z,\gamma z)|\,z \in X_{\G}, \gamma\in\Gamma\backslash\Gamma_{\infty}\big\rbrace,
\end{align}
where $\Gamma_{\infty}$ is the stabilizer of the cusp $\infty$, which is as defined in Section \ref{noncompactmodel}.

\vspace{0.2cm}
\paragraph{\bf{Picard modular cusp forms.}} For $n\geq 2$, let $\G\subset \mathrm{SU}\big((n,1),\C\big)$ be a discrete, torsion-free, cocompact subgroup. For $k\geq 1$, a holomorphic function $f: \Bn \to \C$ is said to be Picard modular cusp form of weight-$k$ with respect to $\G$, if for any $\gamma \in \Gamma$ and $z \in \Bn$, $f$ satisfies the following transformation property:
\begin{align}\label{cuspform:defn}
f\big( \gamma z\big) =\left({Cz+D}\right)^{k} f(z),\,\,\,
\mathrm{where}\,\,
 \gamma=\bigg(\begin{array}{cc} A&B\\C &D\end{array}\bigg) \in \Gamma.
\end{align}
If $\G$ is cofinite, then $f$ has to satisfy the additional condition that $f$ vanishes at all cusps of $\G$. 

\vspace{0.1cm}
So, if $\G\subset \mathrm{SU}\big((n,1),\calo_{K}\big)$ is a torsion-free, finite index subgroup with only one cusp at $\infty$, then for $f$ to be a Picard modular cusp form of weight-$k$ with respect to $\G$, $f$ has to satisfy the transformation property \eqref{cuspform:defn}, and the condition that $f(\infty)=0$. 

\vspace{0.1cm}
The complex vector-space of weight-$k$ Picard modular cusp forms is denoted by $\cals_{k}(\Gamma)$.

\vspace{0.1cm}
For any $f\in \cals_{k}(\Gamma)$, the Petersson norm at a point $z\in X$ is given by the following formula:
\begin{align}
\big|f(z)\big|_{\mathrm{pet}}^{2}:= \big( 1- \big|z\big|^2\big)^{k}\big|f(z)\big|^{2}.
\end{align}

The Petersson norm induces an $L^2$-metric on $\cals_{k}(\Gamma)$, which we denote by $\langle\cdot,\cdot\rangle_{\mathrm{pet}}$, and is known as the Petersson inner-product. For any $f,g\in\cals_{k}(\G)$, the Petersson inner-product is given by the following formula
\begin{align*}
\langle f, g\rangle_{\mathrm{pet}}:=\int_{\calf_{\G}}\big( 1- \big|z\big|^2\big)^{k}f(z)\overline{g(z)}\hypbnvol(z),
\end{align*}
where $\calf$ denotes a fundamental domain of $X_{\G}$

\vspace{0.1cm}
When $\G$ is cocompact, Picard modular cusp forms can be realized as the global section of a line bundle. Let $\mathcal{L}$ denote the line bundle, whose sections are Picard modular cusp forms of weight-$1$. Then, for $k\geq 1$, we have
\begin{align*}
H^{0}\big(X_{\G}, \mathcal{L}^{\otimes k}\big)=\cals_{k}(\Gamma).
\end{align*}

The Petersson norm and the Petersson inner-product can be realized as the point-wise norm and $L^{2}$ metric on $H^{0}\big(X_{\G},\mathcal{L}^{\otimes k}\big)$.

\vspace{0.2cm}
\paragraph{\bf{Bergman kernel.}}
Let $d_{k}$ denote the dimension of the vector-space $\cals_{k}(\Gamma)$, and let $\lbrace f_{1},\ldots f_{d_{k}} \rbrace$ denote an orthonormal basis of $\cals_{k}(\Gamma)$ with respect to the Petersson inner-product. Then, for any $z,w\in \Bn$, the Bergman kernel $\bk(z)$ associated to the vector  space $\cals_{k}(\Gamma)$ is given by the following formula:
\begin{align}\label{bkdefn1}
\bk(z,w):=\sum_{j=1}^{d_{k}}f_{j}(z)\overline{f_{j}(w)},
\end{align}
and when $z=w$, for brevity of notation, we denote $\bk(z,z)$ by $\bk(z)$.  

\vspace{0.1cm}
For a fixed $w\in \Bn$, the Bergman kernel is a holomorphic cusp form of weight-$k$ in  $z$. Similarly, for a fixed $z\in \Bn$, the Bergman kernel is a anti-holomorphic cusp form of weight-$k$ in $w$. The Petersson norm induces the following norm on the Bergman kernel at any $z,w\in \Bn$ with respective lifts $\tilde{z}, \tilde{w}\in\cala^{n+1}$ is given by the following formula
\begin{align*}
\big|\bk(z,w)\big|_{\mathrm{pet}}:=\big(-\big\langle \tilde{z},\tilde{z} \big\rangle_{H}\big)^{k\slash 2}\cdot\big(-\big\langle \tilde{w},\tilde{w} \big\rangle_{H}\big)^{k\slash 2}=\big(1- \big|z\big|^2 \big)^{k\slash2}\cdot \big(1- \big|w\big|^2 \big)^{k\slash2}\cdot\big|\bk(z,w)\big|.
\end{align*}

Now we introduce an alternate expression for the Bergman kernel associated to $\cals_{k}(\Gamma)$, the complex vector-space of Picard cusp forms of the ball domain $\mathbb{B}_n$. For any $k\geq 2n+1$, and $z=(z_1,\ldots,z_n)^{t},\,w=(w_1,\ldots, w_n)^{t}\in \mathbb{B}_n$, the Bergman kernel $\bk(z,w)$ associated to $\cals_{k}(\Gamma)$ has an alternate expression, which is given by the following formula
\begin{align}
&\bk(z,w):=\sum_{\gamma\in\Gamma}\frac{\ck}{\big|\langle z,\gamma w\rangle\big|^{k}(Cw+D)^{k}},\notag\\[0.1cm]
&\mathrm{where} \,\, \gamma=\left(\begin{array}{cc}A&B\\C&D \end{array}\right),\label{bkdefn2}\,\,\mathrm{and}\,\,\langle z, w\rangle:=1-\sum_{j=1}^{n}z_{j}\overline{w}_{j}, 
\end{align}
and $\ck$ is a constant, which satisfies the following estimate
\begin{align}\label{estimateck}
\ck=O_{\G}(k^n).
\end{align}
The implied constant in the above estimate is a constant, which depends only on $\G$, and can be computed explicitly by combining the two definitions of the Bergman kernel, namely equations \eqref{bkdefn1} and \eqref{bkdefn2}, and using the Riemann-Roch theorem.

\vspace{0.1cm}
From \eqref{bkdefn2}, the Petersson norm at any $z\in X_{\G}$ is given by 
\begin{align}\label{eqn1proof1}
\big|\bk(z)\big|_{\mathrm{pet}}=\big(1-\big|z\big|^{2}\big)^{k}\cdot\big|\bk(z)\big|\leq\sum_{\gamma\in\Gamma}\frac{\ck\big(1-\big|z\big|^{2}\big)^{k\slash 2}}{\big|1-\langle z,\gamma z\rangle\big|^{k}}\cdot\frac{\big(1-\big| z\big|^{2}\big)^{k\slash 2}}{\big|Cz+D\big|^{k}}.
\end{align}

For any $\gamma=\left(\begin{array}{cc}A&B\\C&D\end{array}\right)\in\Gamma$, and $z\in X_{\G}$, we have
\begin{align*}
\big(1-\big|\gamma z\big|^{2}\big)=-\big\langle \tilde{\gamma z},\tilde{\gamma z}\big\rangle_{H}=\frac{\big(1-\big|z\big|^{2}\big)}{\big|Cz+D\big|^{2}},
\end{align*}
where $\tilde{\gamma z}$ denotes the lift of $\gamma z$ to $\cala^{n+1}$.

\vspace{0.1cm}
Using the above relation, and combining equations \eqref{coshreln} and \eqref{eqn1proof1}, we derive
\begin{align}\label{coshyprel}
\big|\bk(z)\big|_{\mathrm{pet}}\leq\sum_{\gamma\in\Gamma}\frac{\ck}{\cosh^{k}\big(\dhyp(z,\gamma z)\slash 2\big)}.
\end{align}

\vspace{0.2cm}
\paragraph{\textbf{Picard modular cusp forms on $\mathcal{H}^{2}(\C)$.}}
We now describe the hyperbolic distance,  and Petersson norm of Picard modular cusp forms, in Model (3) of the hyperbolic space $\mathcal{H}^{2}(\C)$. 

\vspace{0.1cm}
For any $z=(z_1,z_2)^{t},w=(w_1,w_2)^t\in\mathbb{B}^2_3$ with respective lifts $\tilde{z}=(z_1,z_2,1)^{t},\tilde{w}=(w_1,w_2,1)\in\cala^{3}$, the hyperbolic distance is given by the following formula
\begin{align}\label{hyp-dist2}
\cosh^{2}\big(\dhyp(z,w)\slash 2\big)=\frac{\langle \tilde{z},\tilde{w}\rangle_3\langle \tilde{w},\tilde{z} \rangle_3}{\langle \tilde{z},\tilde{z}\rangle_3 \langle \tilde{w},\tilde{w} \rangle_3},
\end{align}
where $\langle \cdot,\cdot\rangle_3$ is as given by equation \eqref{innerpro3}. 

\vspace{0.2cm}
For any $k\geq 1$ and $f\in\cals_{k}(\G)$, at any $z\in \mathbb{B}^2_3$ with lift $\tilde{z}\in\cala^{3}$, the Petersson norm is given by the following formula
\begin{align}\label{pet-norm3}
\big|f(z)\big|_{\mathrm{pet}}^{2}:=\big(-\langle \tilde{z},\tilde{z}\rangle_3\big)^{k}\cdot\big|f(z)\big|^{2}=\big(-2 \mathrm{Re}(z_1) - \big|z_{2} \big| ^2\big)^{k}\big|f(z)\big|^{2}.
\end{align}
\section{Asymptotic estimates of the Bergman kernel}\label{asymptoticestimate}
In this section, we state the relevant results from geometric analysis, that are required in proving Theorems \ref{mainthm1} and \ref{mainthm2}, and Corollary \ref{asympcor}. We repeat the exposition from \cite{anil1} and \cite{ab1} for the benefit of the reader. 

\vspace{0.2cm}
\paragraph{\bf{Asymptotic estimates.}} 
Let $(M,\omega)$ be a compact complex manifold of dimension $n$ with natural Hermitian metric $\omega$. Let $\lambda$ and $\ell$, be a vector bundle of rank $r$ and a positive Hermitian holomorphic line bundle on $M$, respectively. Let $\|\cdot\|_{\lambda}$ and $\|\cdot\|_{\ell}$ denote the Hermitian metrics on and $\lambda$ and  $\ell$, respectively.  Furthermore, for any section $s$ of the line bundle $\ell$, the Hermitian metric $\|\cdot\|_{\ell}$ is given by the following formula $\|s(z)\|_{\ell}^{2}:=e^{-\phi(z)}|s(z)|^{2}$, where $\phi(z)$ is a real-valued function.

\vspace{0.1cm}
For any $k\geq 0$, let $\lk:=(\overline{\partial}^{\ast}+\overline{\partial})^{2}$ denote the $\overline{\partial}$-Laplacian acting on smooth sections of  $\lambda\otimes\ell^{\otimes k}$. Let $\hkm(t;z,w)$ denote the smooth kernel of the operator $e^{-\frac{2t}{k}\lk}$. We refer the reader to p.\,2 in \cite{bouche}, for the details regarding the properties which uniquely characterize the heat kernel $\hkm(t;z,w)$. When $z=w \in M$, the heat kernel $\hkm(t;z,z)$ admits the following spectral expansion:
\begin{align}\label{spectralexpn}
\hkm(t;z,z)=\sum_{n\geq 0} e^{-\frac{2t}{k}\alpha_{n}^{k}}\|\varphi_{n}(z)\|_{\lambda\otimes\ell^{\otimes k}}^{2},
\end{align}
where $\lbrace\alpha_{n}^k\rbrace_{n\geq0}$ denotes a set of eigenvalues of $\lk$ 
(counted with multiplicities), and $\lbrace\varphi_{n}\rbrace_{n\geq0}$ denotes the set of associated orthonormal eigenfunctions. 

\vspace{0.1cm}
Let $H^{0}(M,\lambda\otimes\ell^{\otimes k})$ denote the vector space of global holomorphic sections of the vector bundle $\lambda\otimes\ell^{\otimes k}$, and let  $\lbrace s_{i}\rbrace$ denote an orthonormal basis of $H^{0}(M,\lambda\otimes\ell^{\otimes k})$. For any $z\in M$, the following function is called the Bergman kernel associated 
to the vector bundle $\lambda\otimes\ell^{\otimes k}$
\begin{align}\label{bkdefn}
\bkm(z):= \sum_{i}\| s_{i}(z)\|_{\lambda\otimes\ell^{\otimes k}}^{2}.
\end{align}

The above definition is independent of the choice of orthonormal basis of $H^{0}(M,\lambda\otimes\ell^{\otimes k})$. 

\vspace{0.1cm}
It is easy to show that $\mathrm{ker}(\lk)=H^{0}(M,\lambda\otimes\ell^{\otimes k})$. Using which, and combining it with the spectral expansion of the heat kernel $\hkm(t;z,w)$ described in equation \eqref{spectralexpn}, it is easy to see that
\begin{align}\label{hkbkreln}
\lim_{t\rightarrow\infty} \hkm(t;z,z)=\bkm(z).
\end{align}
Let 
\begin{align}\label{curvatureform}
c_{1}(\ell)(z):=\frac{i}{2\pi}\partial\overline{\partial}\phi(z)
\end{align}
denote the curvature form of the line bundle $\ell$ at the point $z\in M$. Let $\alpha_{1},\ldots,\alpha_{n}$ denote the eigenvalues of $\partial\overline{\partial}\phi(z)$ at the point $z\in M$. Then, with notation as above, it follows from \cite[Th.\,1.1]{bouche} that, for any $z\in M$ and $t\in (0,k^{\varepsilon})$, for a given $\varepsilon>0$ not depending on $k$, we have
\begin{align}\label{boucheeqn1}
\lim_{k\rightarrow\infty}\frac{1}{k^{n}}\hkm(t;z,z)=r\cdot\prod_{j=1}^{n}\frac{\alpha_{j}}{(4\pi)^{n}\sinh(\alpha_{j}t)}, 
\end{align}
where $r$ denotes the rank of the vector bundle $\lambda$, and the convergence of the above limit is uniform in $z$. 

\vspace{0.1cm}
Using equations \eqref{hkbkreln} and \eqref{boucheeqn1}, in \cite[Th.\,2.1]{bouche}, Bouche derives the following equality
\begin{align}\label{boucheeqn2}
\lim_{k\rightarrow\infty}\frac{1}{k^{n}}\bkm(z)=r\cdot \big|\mathrm{det}_{\omega}\big(c_{1}(\ell)(z)\big)\big|,
\end{align}
and the convergence of the above limit is uniform in $z\in M$. 

\vspace{0.1cm}
For $n\geq 2$, let $\Gamma\subset \mathrm{SU}\big((n,1),\C\big)$ be a discrete, torsion-free, cocompact subgroup. For $k\geq 1$, let $f$ be a Picard modular cusp form of weight-$k$. Then, $f$ is a global section of $\mathcal{L}^{\otimes k}$, where $\mathcal{L}$ is a holomorphic line bundle, whose sections are Picard modular cusp forms of weight-$1$. 

\vspace{0.1cm}
For $k\geq 1$ and $z\in X_{\G}$, from the definition of the Bergman kernel $\mathcal{B}_{X_{\G}}^{\mathcal{L}^{\otimes k}}(z)$ for the line bundle $\mathcal{L}^{\otimes k}$ from equation \eqref{bkdefn}, and from the discussion in Section \ref{picardcuspforms}, we have 
\begin{align}\label{remeqn2}
\mathcal{B}_{X_{\G}}^{\mathcal{L}^{\otimes k}}(z)=\big|\bk(z)\big|_{\mathrm{pet}}. 
\end{align}

Furthermore, for any $f\in H^{0}(X,\mathcal{L})$ and $z\in X$, the Hermitian metric $\|\cdot\|_{\mathcal{L}}$ at $z\in X_{\G}$, is given by the following formula
\begin{align}\label{metriconL}
\|f\|_{\mathcal{L}}^{2}(z):=\big(1-|z|^{2}\big)\big|f(z)\big|^{2}.
\end{align}

\vspace{0.1cm}
We now prove Theorem 1. 

\vspace{0.1cm}
\begin{thm}\label{mainthm1}
With notation be as above, for $n\geq2$, let $\Gamma\subset \mathrm{SU}\big((n,1),\C\big)$ be a discrete, torsion-free, cocompact subgroup. Then, for any $z\in X_{\G}$,  we have 
\begin{align}\label{eq1:mainthm1}
\lim_{k\rightarrow \infty}\frac{1}{k^n}\big|\bk(z)\big|_{\mathrm{pet}}=  \frac{1}{(4\pi)^{n}},
\end{align}
and the convergence of the above limit is uniform in $z\in X_{\G}$.
\end{thm}
\begin{proof}
From equation \eqref{remeqn2},  for any $n\geq 2$, and $k\geq 1$, and $z\in X_{\G}$, we have
\begin{align}\label{proofeqn0}
\big|\bk(z)\big|_{\mathrm{pet}}=\mathcal{B}_{X_{\G}}^{\mathcal{L}^{\otimes k}}(z).
\end{align}

Applying estimate \eqref{boucheeqn2} to the complex manifold $X$ with its natural Hermitian metric $\hypbn(z)$, and the line bundle $\mathcal{L}^{\otimes k}$, we deduce that 
\begin{align}\label{proofeqn1}
\lim_{k\rightarrow \infty}\frac{1}{k^{n}}\mathcal{B}_{X_{\G}}^{\mathcal{L}^{\otimes k}}(z)=\big|\mathrm{det}_{\hypbnvol}\big(c_{1}(\mathcal{L})(z)\big)\big|, 
\end{align}
and the convergence of the above limit is uniform in $z\in X$. Now from the definition of the Hermitian metric $\|\cdot\|_{\mathcal{L}}$ on the line bundle $\mathcal{L}$ given by equation \eqref{metriconL}, and the definition of the curvature-form given by equation \eqref{curvatureform}, and the definition of the hyperbolic metric $\hypbn(z)$ given by equation \eqref{defnhypmetric}, we have
\begin{align}\label{proofeqn2}
c_{1}(\mathcal{L})(z)=-\frac{i}{2\pi}\partial\overline{\partial}\log\big(1-|z|^{2}\big)=\frac{1}{4\pi}\hypbn(z),
\end{align}
which implies that 
\begin{align}\label{proofeqn3}
\big|\mathrm{det}_{\hypbnvol}\big(c_{1}(\mathcal{L})(z)\big)\big|=\frac{1}{(4\pi)^{n}}. 
\end{align}
Hence, combining equations \eqref{proofeqn0},  \eqref{proofeqn1}, \eqref{proofeqn2}, and \eqref{proofeqn3}, we have
\begin{align*}
\lim_{k\rightarrow \infty}\frac{1}{k^{n}}\big|\bk(z)\big|_{\mathrm{pet}}=\frac{1}{(4\pi)^{n}},
\end{align*}
and the convergence of the above limit is uniform in $z\in X_{\G}$. This completes the proof of the theorem. 
\end{proof}

\vspace{0.1cm}
As an immediate consequence of Theorem \ref{mainthm1}, we prove the following Corollary, which can be thought of as an average version of Arithmetic Quantum Unique Ergodicity. 

\vspace{0.1cm}
\begin{cor}\label{asympcor}
With notation be as above, for $n\geq 2$, let $\Gamma\subset \mathrm{SU}\big((n,1),\C\big)$ be a discrete, torsion-free, cocompact subgroup. Then, for any  $z\in X_{\G}$, we have
\begin{align*}
 \lim_{k\rightarrow \infty}\frac{1}{d_{k}}\big|\bk(z)\big|_{\mathrm{pet}}\hypbnvol(z)=\shyp(z), 
\end{align*}
and the convergence of the above limit is uniform in $z\in X_{\G}$.
\begin{proof}
From equation \eqref{eq1:mainthm1} in Theorem \ref{mainthm1}, for $n\geq 2$, and $z\in X_{\G}$, we have
\begin{align}\label{cor1eqn1}
\lim_{k\rightarrow \infty}\frac{1}{d_{k}}\big|\bk(z)\big|_{\mathrm{pet}}\hypbnvol(z) &=\lim_{k\rightarrow \infty}\frac{k^n}{d_{k}}\cdot\lim_{k\rightarrow \infty}\frac{1}{k^n}
\big|\bk(z)\big|_{\mathrm{pet}}\hypbnvol(z)\notag\\ 
&= \lim_{k\rightarrow \infty}\frac{k^n}{d_{k}}\cdot\frac{1}{(4\pi)^{n}}\cdot \hypbnvol(z).
\end{align}
Observe that 
\begin{align*}
d_{k}=\int_{X_{\G}}\big|\bk(z)\big|_{\mathrm{pet}}(z)\hypbnvol(z).
\end{align*}
As the limit in equation \eqref{eq1:mainthm1} is uniformly convergent, using the above equation, we derive
\begin{align}\label{cor1eqn2}
 \lim_{k\rightarrow \infty}\frac{d_{k}}{k^n} &=\lim_{k\rightarrow \infty}\frac{1}{k^n}\int_{X}\big|\bk(z)\big|_{\mathrm{pet}}\hypbnvol(z) \notag \\
 &= \int_{X_{\G}}\lim_{k\rightarrow \infty}\frac{1}{k^{n}}\big|\bk(z)\big|_{\mathrm{pet}}\hypbnvol(z) \notag \\
 &=\frac{1}{(4\pi)^{n}}\cdot\vx(X).
\end{align}
Combining equations \eqref{cor1eqn1} and \eqref{cor1eqn2}, we compute
\begin{align*}
\lim_{k\rightarrow \infty}\frac{1}{d_{k}}\big|\bk(z)\big|_{\mathrm{pet}}\hypbnvol(z) &=\lim_{k\rightarrow \infty}\frac{k^n}{d_{k}}\cdot\frac{1}{(4\pi)^{n}}\cdot \hypbnvol(z) \\
&=\frac{1}{\vx(X)}\hypbnvol(z)=\shyp(x),
\end{align*}
and the convergence of the limit is uniform in $z\in X_{\G}$. This completes the proof of the corollary.
\end{proof}
\end{cor}

\vspace{0.1cm}
\begin{rem}
From \cite[Th.\,1.1]{dai}, it is easy to show that
\begin{align}\label{rateofconv}
\lim_{k\rightarrow \infty}\frac{1}{d_{k}}\big|\bk(z)\big|_{\mathrm{pet}}(z)\hypbn(z)= \shyp(z)+O_{X}(k^{-n}),
\end{align}
where the implied constant depends only on the Picard modular variety $X$. A careful analysis of the theorem cited above should enable one to prove that the implied constant in equation \eqref{rateofconv} is independent of $X$.
\end{rem}

\vspace{0.1cm}
\begin{rem}
As stated earlier, as in \cite{ab1}, Theorem \ref{mainthm1} and Corollary \ref{asympcor} easily extend to Bergman kernels associated to vector-valued Picard modular cusp forms.  However for brevity of exposition, and for the notation to be in sync with the second theorem, we restrict ourselves to the case of scalar-valued Picard modular cusp forms.
\end{rem}

\vspace{0.2cm}
\section{Quantitative estimates of the Bergman kernel in the cocompact setting}\label{results2}
In this section, for $n\geq 2,$, we prove quantitative estimates of the Bergman kernel, when $\G\subset \mathrm{SU}\big((n,1),\C\big)$ is a discrete, torsion-free, cocompact subgroup. 

\vspace{0.1cm}
With notation as above, for $n\geq 2$, and any $z,w\in X_{\G}$, put 
\begin{align}\label{defncountingfn}
\cals_{\Gamma}(z,w;\delta):=&\,\big\lbrace \gamma\in \Gamma| \,\dhyp(z,\gamma w)\leq \delta\big\rbrace, \notag\\
N_{\Gamma}(z,w;\delta):=&\,\mathrm{card}\,\cals_{\Gamma}(z,w;\delta).
\end{align}

\vspace{0.1cm}
Let $\mathrm{vol}\big(B(z,r)\big)$ denote the volume of a geodesic ball of radius $r$, centred at any $z\in\mathbb{B}^{n}$. Then, from arguments from elementary hyperbolic geometry, we have
\begin{align}\label{hypball}
\mathrm{vol}\big(B(z,r)\big)=\frac{4\pi\sinh^{2n}\big(r\slash 2\big)}{n!}.
\end{align}

In the following lemma, we estimate the counting function $N_{\Gamma}$.

\vspace{0.1cm}
\begin{lem}\label{countingfnestimateslem}
With notation as above, let $z,w\in X_{\G}$. Then, for any $\delta >0$, we have the following estimate
\begin{align}\label{countingfnestimate1}
N_{\Gamma}(z,w;\delta)\leq \frac{4\pi\sinh^{2n}\big((2\delta+\rx)\slash4\big)}{n!\sinh^{2n}\big(\rx\slash 4\big)};
\end{align} 
let $f(\rho)$ be any positive, smooth, monotonically decreasing function defined on $\R_{>0}$. Then, for any $\delta> \rx\slash 2$, assuming that all the involved integrals exist, we have the following inequality
\begin{align}\label{countingfnestimate2}
\int_{0}^{\infty}f(\rho)dN_{\Gamma}(z,w;\rho)\leq \int_{0}^{\delta}f(\rho)dN_{\Gamma}(z,w;\rho)+f(\delta)\frac{4\pi \sinh^{2n}\big((2\delta+\rx)\slash 4\big)}{n!\sinh^{2n}\big(\rx\slash 4\big)}+\notag
\\[0.13cm]\frac{4\pi}{(n-1)!\sinh^{2n}\big(\rx\slash 4\big)}\int_{\delta}^{\infty}f(\rho)\sinh^{2n-1}\big((2\rho+\rx)\slash 4\big)\cosh\big((2\rho+\rx)\slash4\big)d\rho,
\end{align}
where $c_{n}$ is as defined in equation \eqref{hypball}.
\begin{proof}
The main idea of the proof is the same as in Lemma $4$ in \cite{jl}. We estimate  $N_{\Gamma}(z,w;\rho)$ by the number of disjoint balls of radius $\rx\slash 2$, which can be embedded in a ball of  radius $\rho+\rx\slash 2$. Observe that our definition of injectivity radius $\rx$ is two times the definition employed in \cite{jl}, and our estimates take that into account.  

\vspace{0.1cm}
Inequality \eqref{countingfnestimate1} follows directly from equation \eqref{hypball}.

\vspace{0.1cm}
From similar arguments, it is easy to see that for any $\rho\geq\delta>\rx\slash 2$, we find
\begin{align*}
N_{\Gamma}(z,w;\rho)\leq N_{\Gamma}(z,w;\delta)+\frac{\sinh^{2n}\big((2\rho+\rx)\slash 4\big)-\sinh^{2n}\big((2\delta-\rx) \slash 4\big)}{\sinh^{2n}\big(\rx\slash 4\big)}.
\end{align*}
Then, inequality \eqref{countingfnestimate2} follows from similar arguments as in Lemma $4$ in \cite{jl}.  
\end{proof}
\end{lem}

\vspace{0.1cm}
Using the above estimate for the counting function $N_{\G}$, we now estimate the Bergman kernel. 

\vspace{0.1cm}
\begin{thm}\label{mainthm2}
With notation as above, and for $n\geq 2$, and $k\geq 2n+2$, we have the following estimate
\begin{align}\label{estmainthm2}
&\sup_{z\in X_{\G}}\big|\bk(z)\big|_{\mathrm{pet}}\leq \notag\\[0.12cm]&\ck+ \frac{\ck\cosh^{2n}\big(\rx\slash 4\big)) }{\big(k-2n-1\big)\sinh^{2n}\big(\rx\slash 4\big)}
+\frac{\ck\sinh^{2n}\big(5\rx\slash 8\big)}{\sinh^{2n}\big(\rx\slash 4\big)\cosh^{k}\big(3\rx\slash 8\big)},
\end{align}
where $\rx$ is the injectivity radius of $X_{\G}$, which is as defined in equation \eqref{ircom}, and $\ck$ is as described in equation \eqref{bkdefn2}. 
\end{thm}
\begin{proof}
For $n\geq 2$, and $k\geq 2n+2$, and any $z\in X_{\G}$, from inequality \eqref{coshyprel} we have
\begin{align}\label{eqn1proof2}
\big|\bk(z)\big|_{\mathrm{pet}}\leq\sum_{\gamma\in\Gamma}\frac{\ck}{\cosh^{k}\big(\dhyp(z,\gamma z)\slash 2\big)}.
\end{align}
Now, as $\cosh^{k}\big(\rho\slash 2\big)$ is a positive, smooth, monotonically decreasing function on $\R_{>0}$,  substituting $\delta=3\rx\slash 4$ in estimates \eqref{countingfnestimate1} and \eqref{countingfnestimate2} from Lemma \ref{countingfnestimateslem}, we compute
\begin{align}\label{eqn2proof2}
\sum_{\gamma\in\Gamma}\frac{\ck}{\cosh^{k}\big(\dhyp(z,\gamma z)\slash 2\big)}=\int_{0}^{\infty}\frac{\ck dN_{\Gamma}(z,\gamma z;\rho)}{\cosh^{k}\big(\dhyp(z,\gamma z)\slash 2\big)}\leq \notag\\[0.13cm]
\int_{0}^{\frac{3\rx}{ 4}}\frac{\ck dN_{\Gamma}(z,\gamma z;\rho)}{\cosh^{k}\big(\dhyp(z,\gamma z)\slash 2\big)}+\frac{c_{n}\ck\sinh^{2n}\big(5\rx\slash 8\big)}{\sinh^{2n}\big(\rx\slash 4\big)\cosh^{k}\big(3\rx\slash 8\big)}+\notag\\[0.13cm]\frac{nc_n\ck}{\sinh^{2n}\big(\rx\slash 4\big)}\int_{\frac{3\rx}{4}}^{\infty}\frac{\sinh^{2n-1}\big((2\rho+\rx)\slash 4\big)\cosh\big((2\rho+\rx)\slash4\big)}{\cosh^{k}\big(\rho\slash 2\big)}d\rho.
\end{align}
We now estimate the first term on the right hand-side of the above inequality. From the definition of injectivity radius $\rx$ in \eqref{ircom}, we have the following estimate for the first term on the right hand-side of inequality \eqref{eqn2proof2}
\begin{align}\label{eqn3proof2}
\int_{0}^{\frac{3\rx}{ 4}}\frac{\ck dN_{\Gamma}(z,\gamma z;\rho)}{\cosh^{k}\big(\dhyp(z,\gamma z)\slash 2\big)}=\ck.
\end{align}
Observe that for any $\rho\geq 0$, we have
\begin{align*}
\cosh\big((2\rho+\rx)\slash 4\big)\leq 2\cosh\big(\rx\slash 4\big)\cosh\big(\rho\slash 2\big), \\[0.1cm]
\sinh^{2n-1}\big((2\rho+ \rx)\slash4\big)\leq\sinh(\rho)\leq \big(2\cosh\big(\rx\slash 4\big)\cosh\big(\rho\slash 2\big)\big)^{2n-1}=\notag\\[0.1cm]2^{2n-1}\cosh^{2n-1}\big(\rx\slash 4\big)\cosh^{2n-1}\big(\rho\slash 2\big).
\end{align*}
Using the above two inequalities, for $k\geq 2n+2$, we have the following estimate for  the third term on  the right hand-side of inequality \eqref{eqn2proof2}
\begin{align}\label{eqn4proof2}
\frac{nc_n\ck}{\sinh^{2n}\big(\rx\slash 4\big)}\int_{\frac{3\rx}{4}}^{\infty}\frac{\sinh^{2n-1}\big((2\rho+\rx)\slash 4\big)\cosh\big((2\rho+\rx)\slash4\big)}{\cosh^{k}\big(\rho\slash 2\big)}d\rho\leq \notag\\[0.14cm]
\frac{2^{2n}nc_n\ck \cosh^{2n}\big(\rx\slash 4\big) }{\sinh^{2n}\big(\rx\slash 4\big)}\int_{\frac{3\rx}{4}}^{\infty}\frac{d\rho}{\cosh^{k-2n}\big(\rho\slash 2\big)}
\leq \frac{2^{2n}nc \ck\cosh^{2n}\big(\rx\slash 4\big)) }{\big(k-2n-1\big)\sinh^{2n}\big(\rx\slash 4\big)},
\end{align}
where $c$ is a universal constant. 

\vspace{0.1cm}
The proof of estimate \eqref{estmainthm2}  follows form combining inequalities \eqref{eqn1proof2}--\eqref{eqn4proof2}, and absorbing all the constants depending upon $n$ into $C_{\G,k}$.
\end{proof}
\begin{cor}\label{cor3}
With notation as above, for $n\geq 2$ and $k\geq 2n+1$, we have the following estimate
\begin{align} \label{eqn:cor3}
 \sup_{z\in X_{\G}} \big|\bk(z)\big|_{\mathrm{pet}} = O_{\Gamma}\big(k^{n}\big).
 \end{align}
\end{cor}
\begin{proof}
The proof of the corollary follows directly from combining Theorem \ref{mainthm2} and \eqref{estimateck}.
\end{proof}

\vspace{0.1cm}
\begin{rem}
From asymptotic estimate \eqref{eq1:mainthm1}, we can conclude that estimate \eqref{eqn:cor3} is optimally derived. 
\end{rem}
\section{Quantitative estimates of the Bergman kernel in the cofinite setting}\label{noncomestimate}
For the rest of the article, $\Gamma \subset \mathrm{SU}\big((2,1),\calo_K\big)$ denotes a torsion-free, finite index subgroup with only one cusp at $\infty$, where $K$ is a totally imaginary number field. We recall from subsection \ref{noncompactmodel}, the complex hyperbolic $2-$space $\mathcal{H}^2(\C)$ can be identified with 
\begin{align*}
\mathbb{B}^2:=\big\lbrace z=\big(z_1,z_2\big) \in \mathbb C^2 \big|\, \big|z_1\big|^{2}  + \big|z_2\big|^2 < 1 \big\rbrace.
\end{align*}

As discussed in Sections \ref{ballmodel} and \ref{noncompactmodel}, the resulting quotient space $X_{\G}:=\Gamma\backslash \mathbb{B}^2$ is a noncompact complex manifold of dimension $2$, of finite hyperbolic volume, and with one cusp. 

\vspace{0.1cm}
In this section, we prove the quantitative estimates of the Bergman kernel associated to Picard modular cusp forms of weight-$k$, with respect to $\G$. For this, we go back and forth between the three models of the complex hyperbolic $2$-space $\mathcal{H}^2(\C)$, which are as discussed in Section \ref{noncompactmodel}. 

\vspace{0.1cm}
In the following lemma, we determine the points, where the Bergman kernel attains its maxima. 

\vspace{0.1cm}
\begin{lem}\label{supnormlemma}
With notation as above, for any $k\geq 1$ and $z\in\mathbb{B}^2_3$, the Petersson norm of the Bergman kernel $\big|\bk(z)\big|_{\mathrm{pet}}$ attains its maximum value on
\begin{align}\label{supnormlemma:eqn}
\cals := \big\lbrace z \in \mathbb{B}^2_3\big| \,\, \mathrm{Re} (z_1) = -\frac{k}{4\pi},\, z_2=0  \big\rbrace.
\end{align}
\end{lem}
\begin{proof}
From Section \ref{noncompactmodel}, recall that $\mathbb{B}^2_3$, Model (2) of the hyperbolic $2$-space $\mathcal{H}^2(\C)$ is given by the following equation 
\begin{align*}
\mathbb{B}^2_3:=\lbrace z=(z_1,z_2)^{t} \in \mathbb C^2 \big|\,\, 2 \mathrm{Re}(z_1) + \big|z_2\big| ^2  < 0 \rbrace.
\end{align*}

\vspace{0.1cm}
From the definition of the Petersson norm for a Picard modular cusp form from equation \eqref{pet-norm3}, we have 
\begin{align}\label{eqnfourier}
\big|f(z)\big|_{\mathrm{pet}}^{2}=\big(-2 \mathrm{Re}(z_1 )- \big|z_2\big|^{2}\big)^{k} \big|f(z)\big|^2 =
\big(-2\mathrm{Re}(z_1)-\big|z_2\big|^2\big)^k \cdot\bigg| \dfrac{f(z)}{e^{2\pi \mathrm{Re}(z_1)}}\bigg|^2\cdot e^{4\pi \mathrm{Re}(z_1)} .
\end{align}

The first term of the expression is subharmonic, and bounded in any open region of  $\mathbb{B}^2_3$, and will attain its maximum at the boundary of the region.

\vspace{0.1cm}
We now compute the region where the following function
\begin{equation*}
P(z)=\big(-2\mathrm{Re}(z_1)-\big|z_2\big|^2\big)^k \cdot e^{4\pi \mathrm{Re}(z_1)}
\end{equation*}
attains its maximum. 

\vspace{0.1cm}
For $z=(z_1=x_1+ i y_1,\,z_2=x_2 + iy_2)^{t}\in \mathbb{B}^2_3$, we have 
\begin{equation*}
P(z) = \big(-2x_1-x_2^2-y_2^2\big)^k  e^{4\pi x_1}
\end{equation*} 

Now we look at the partial derivative of $P$ with respect to $x_1,\,x_2,\,y_1,\,y_2$.
 \begin{align*}
 &\frac{\partial P(z)}{\partial x_1} = \big(-2x_1-x_2^2-y_2^2 \big)^{k-1}\big(-2k-4\pi(2x_1+x_2^2+y_2^2)\big)e^{4\pi x_1}; \,\, \,\, \,\frac{\partial P(z)}{\partial y_1} = 0;\\[0.12cm]
 &\frac{\partial P(z)} { \partial x_2} = -2kx_2 \big(-2x_1-x_2^2-y_2^2 \big)^{k-1}e^{4\pi x_1};\,  \,\,\,\,\frac{\partial P(z)}{ \partial y_2} = -2ky_2 \big(-2x_1-x_2^2-y_2^2\big)^{k-1}e^{4\pi x_1}.
 \end{align*}

Now equating the above quantities to zero, we obtain the points of extrema as
\begin{align*}
\tilde{ z_2} = 0,\hspace{0.15cm} \mathrm{Re}(\tilde{z_1}) = -\frac{k}{4\pi}.
\end{align*}

Further investigation show that the obtained points are indeed points of maxima, which completes the proof of the lemma.
 \end{proof}

\vspace{0.1cm}
Using lemma \ref{supnormlemma}, we estimate the Petersson norm of the Bergman kernel in the following theorem.

\vspace{0.1cm}
\begin{thm}\label{mainthm3}
With notation as above, for any $k\geq 6$, we have the following estimate
\begin{align}\label{mainthm3:eqn}
\sup_{z\in X_{\G}}\big|\bk(z)\big|_{\mathrm{pet}}\leq \ck+ \frac{\ck\cosh^4\big(\rx\slash 4\big)) }{\big(k-5\big)\sinh^{4}\big(\rx\slash 4\big)}
+\frac{\ck\sinh^{4}\big(5\rx\slash 8\big)}{\sinh^{4}\big(\rx\slash 4\big)\cosh^{k}\big(3\rx\slash 8\big)} +\notag\\[0.13cm]\frac{\sqrt{\pi}\G(k\slash 2-1)\G(k-3\slash 2)}{\G(k\slash 2)\G(k-1)} \ck k^{3\slash2 }.
\end{align}
where $\rx$ is the injectivity radius of $X$, which is as defined in equation \eqref{irnoncomp}, and $\ck$ is a  constant satisfies the following estimate 
\begin{align}\label{ck2}
C_{\G,k}=O_{\G}\big(k^{2}\big).
\end{align}
\end{thm}
\begin{proof}
With notation as above, using inequality \eqref{coshyprel}, we derive 
\begin{align} \label{mainthm3-eqn1}
\sup_{z\in X_{\G}}\big|\bk(z)\big|_{\mathrm{pet}}\leq\sup_{z\in  \mathbb{B}^2}\sum_{\gamma\in\Gamma \backslash \Gamma_{\infty}}\frac{\ck}{\cosh^{k}\big(\dhyp(z,\gamma z)\slash 2\big)}+\sup_{z\in  \mathbb{B}^2}\sum_{\gamma\in\Gamma_{\infty}}\frac{\ck}{\cosh^{k}\big(\dhyp(z,\gamma z)\slash 2\big)}.
 \end{align}

Employing similar arguments as the ones used to prove estimate \eqref{mainthm2}, we have the following estimate for the first term of the right-hand side of the above inequality
\begin{align}\label{mainthm3-eqn2}
&\sup_{z\in  \mathbb{B}^2}\sum_{\gamma\in\Gamma \backslash \Gamma_{\infty}}\frac{\ck}{\cosh^{k}\big(\dhyp(z,\gamma z)\slash 2\big)}\notag \leq \\[0.12cm] &\ck+ \frac{\ck\cosh^4\big(\rx\slash 4\big)) }{\big(k-5\big)\sinh^{4}\big(\rx\slash 4\big)}+\frac{\ck\sinh^{4}\big(5\rx\slash 8\big)}{\sinh^{4}\big(\rx\slash 4\big)\cosh^{k}\big(3\rx\slash 8\big)},
\end{align}
where $\rx$, the injectivity radius of $X$, is as defined in equation \eqref{irnoncomp}, and the constant $\ck$ satisfies estimate \eqref{ck2}.

\vspace{0.1cm}
We now estimate the second term on the right-hand side of inequality \eqref{mainthm3-eqn1}. From Lemma \ref{supnormlemma}, and using the fact that $\G_{\infty}^{3}\subset\G_{0,\infty}^{3}$, we make the following inference
\begin{align}\label{mainthm3-eqn3}
\sup_{z\in  \mathbb{B}^2}\sum_{\gamma \in \Gamma_{\infty}} \frac{C_{\Gamma,k}}{\cosh^{k}\big(\dhyp(z,\gamma z)\slash 2\big) }  =\sup_{z\in \mathbb{B}^{2}_{3}} \sum_{\gamma \in \Gamma^3_{\infty}} \frac{C_{\Gamma,k}}{\cosh^{k}\left(\dhyp(z,\gamma z)\slash 2\right) }=\notag\\[0.12cm]\sup_{z\in \cals} \sum_{\gamma \in \Gamma^3_{\infty}} \frac{C_{\Gamma,k}}{\cosh^{k}\left(\dhyp(z,\gamma z)\slash 2\right) }\leq \sup_{z\in \cals} \sum_{\gamma \in \Gamma^3_{0,\infty}} \frac{C_{\Gamma,k}}{\cosh^{k}\left(\dhyp(z,\gamma z)\slash 2\right) },
\end{align}
where $\G_{\infty}^{3}$, $\G_{0,\infty}^{3}$ are as defined in equations \eqref{gammainfty}, \eqref{B3stabilizer}, respectively, and $\cals$ is as described in equation \eqref{supnormlemma:eqn}.

\vspace{0.1cm}
For any $z\in  \mathbb{B}^2_{3}$ with lift $\tilde{z}\in\cala^{3}$, using relation \eqref{hyp-dist2}, we have
\begin{align}\label{mainthm3-eqn4}
\sum_{\gamma \in \Gamma^3_{0,\infty}} \frac{C_{\Gamma,k}}{\cosh^{k}\big(\dhyp(z,\gamma z)\slash 2\big)} 
&=\sum_{\gamma \in \Gamma^3_{0,\infty}} C_{\Gamma,k}\; \frac{\langle \tilde{z},\tilde{z} \rangle_3^{k/2}\langle \gamma \tilde{z},\gamma \tilde{z} \rangle_3 ^{k/2}}{\langle \tilde{z},\gamma \tilde{z} \rangle_3 ^{k/2} \langle \gamma \tilde{z},\tilde{z} \rangle_3 ^{k/2}},
\end{align}
where the inner-product $\langle \cdot,\cdot\rangle_{3}$ is as described in equation \eqref{innerpro3}.

\vspace{0.1cm}
For any $\gamma \in \Gamma^3_{0,\infty}$ and  $z=(z_1,z_2)^t\in  \mathbb{B}^2_{3}$ with lift $\tilde{z}\in\cala^{3}$, and $\gamma\in\G_{0,\infty}$, we compute
\begin{align*}
&\langle \tilde{z}, \tilde{z} \rangle_3 = \langle \gamma \tilde{z}, \gamma \tilde{z} \rangle_3 = 2 \mathrm{Re}(z_1) + \big|z_2\big|^2;\\[0.12cm]
&\langle \gamma \tilde{z} , \tilde{z} \rangle_3 = 2 \mathrm{Re}(z_1) + \big|z_2\big|^2 + \alpha\overline{z}_2 - \overline{\alpha} z_2 - \frac{\big|\alpha\big|^2}{2} + i \beta;\notag \\[0.12cm]& \langle \tilde{z} ,\gamma \tilde{z} \rangle_3 = 2 \mathrm{Re}(z_1) + \big|z_2\big|^2 - \alpha \overline{z}_2 + \overline{\alpha} z_2 - \frac{\big|\alpha\big|^2}{2} -i \beta.
\end{align*}
where $\alpha \in \C$ and $\beta \in \R$.

\vspace{0.1cm}
Using the above equations, for any $z \in \cals$ with lift $\tilde{z}\in\cala^{3}$, and $\gamma\in\G_{0,\infty}$, we compute 
\begin{align}\label{mainthm3-eqn5}
\langle \tilde{z}, \tilde{z} \rangle_3=\langle \gamma\tilde{z}, \gamma\tilde{z} \rangle_3 = -\frac{k}{2\pi};\;\;\;
\langle \gamma \tilde{z} , \tilde{z} \rangle_3 = -\frac{k}{2\pi} - \frac{|\alpha|^2}{2} + i \beta; \;\;\;
\langle \tilde{z} ,\gamma \tilde{z} \rangle_3= -\frac{k}{2\pi}- \frac{|\alpha|^2}{2} - i \beta.
\end{align}
where $\alpha \in \C$ and $\beta \in \R$.

\vspace{0.1cm}
Put
\begin{equation}\label{mainthm3-eqn6}
\mathrm{ L} = \bigg\lbrace(\alpha, \beta) \in \C \times \R\bigg|\,\gamma = \begin{pmatrix} 
1 & -\overline{\alpha} & - \frac{|\alpha|^2}{2} + i\beta \\0 & 1 & \alpha \\0 & 0  & 1 \end{pmatrix}   \in \Gamma^3_{(0,\infty)}\bigg\rbrace,
\end{equation}
which defines a discrete lattice in $\C \times \R$.

\vspace{0.1cm}
Combining equations \eqref{mainthm3-eqn4}--\eqref{mainthm3-eqn6}, for any $z\in\cals$, we derive 
\begin{align}\label{mainthm3-eqn7}
\sum_{\gamma \in \Gamma^3_{0,\infty}} \frac{C_{\Gamma,k}}{\cosh^{k}\big(\dhyp(\tilde{z},\gamma \tilde{z})\slash 2\big)}  =\sum_{(\alpha,\beta ) \in \mathrm{L}} 
\frac{C_{\Gamma,k}\big(k\slash 2\pi\big)^{k}}{\big|k\slash 2\pi +|\alpha|^2\slash 2 + i \beta\big|^{k}}.
\end{align}

Approximating the above summation over the discrete lattice $\mathrm{L}$ by integrals, we compute 
\begin{align}\label{mainthm3-eqn8}
&\sum_{(\alpha,\beta ) \in \mathrm{L}} \frac{C_{\Gamma,k}\big(k\slash 2\pi\big)^{k}}{\big|k\slash 2\pi +
|\alpha|^2\slash 2 + i \beta\big|^{k}}  \leq\notag\\[0.12cm] &\int_{-\infty}^{\infty} \int_{0}^{\infty}\int_{0}^{2\pi} \frac{C_{\Gamma,k}\big(k\slash 2\pi\big)^{k}d\beta drd\theta}{
\big(\big(k\slash 2\pi +r^2\slash 2\big)^{2} +\beta^{2}\big)^{k\slash 2}}= \int_{0}^{\infty}\int_{-\infty}^{\infty} \frac{2\pi C_{\Gamma,k}\big(k\slash 2\pi\big)^{k}drd\beta}{
\big(\big(k\slash 2\pi +r^2\slash 2\big)^{2} +\beta^{2}\big)^{k\slash 2}},
\end{align}
where $\alpha= r e^{i \theta}$.

\vspace{0.1cm}
Now we estimate the integral 
\begin{align*}
 \int_{0}^{\infty}  \int_{-\infty}^{\infty} \frac{drd\beta }{\big(\big(k\slash 2\pi +r^2\slash 2\big)^{2} +\beta^{2}\big)^{k\slash 2}}.
\end{align*}

Substituting
\begin{equation*}
\eta =\beta \slash \big(k\slash 2\pi+r^2\slash 2\big),
\end{equation*}

and using formula 3.251.2 from \cite{grad} in the above integral, we compute 
\begin{align}\label{mainthm3-eqn9}
 \int_{0}^{\infty} \int_{-\infty}^{\infty}\frac{drd\beta }{\big(\big(k\slash 2\pi +r^2\slash 2\big)^{2} +\beta^{2}\big)^{k\slash 2}}=
\frac{\sqrt{\pi}\Gamma(k/2-1/2)}{\Gamma(k/2)}  \int_{0}^{\infty}\frac{dr }{\big(k\slash 2\pi+ r^2\slash2\big)^{k-1}}. 
\end{align}

Substituting $\omega= r\sqrt{{\pi}\slash {k}}$, and applying formula 3.251.2 from \cite{grad} in the above integral, we now compute
\begin{align} \label{mainthm3-eqn9}
 \int_{0}^{\infty}\frac{dr }{\big(k\slash 2\pi+ r^2\slash2\big)^{k-1}}=\sqrt{2}\bigg(\frac{2\pi}{k}\bigg)^{k-3\slash 2}\int_{0}^{\infty} \frac{d\omega }{(1+ \omega^2)^{k-1}} =
\frac{(2\pi)^{k-1}\Gamma(k -3\slash2)}{k^{k-3\slash 2}\Gamma(k-1)} .
\end{align}

For $k\geq 6$, combining equation \eqref{mainthm3-eqn3} with computations \eqref{mainthm3-eqn7}--\eqref{mainthm3-eqn9}, we arrive at the following inequality
\begin{align}\label{mainthm3-eqn10}
\sup_{z\in  \mathbb{B}^2}\sum_{\gamma \in \Gamma_{\infty}} \frac{C_{\Gamma,k}}{\cosh^{k}\big(\dhyp(z,\gamma z)\slash 2\big) }  \leq \frac{\sqrt{\pi}\G(k\slash 2-1)\G(k-3\slash 2)}{\G(k\slash 2)\G(k-1)} \ck k^{3\slash2 }.
\end{align}
 
Combining inequalities \eqref{mainthm3-eqn1}, \eqref{mainthm3-eqn2}, and \eqref{mainthm3-eqn10} completes the proof of the theorem.
\end{proof}

\vspace{0.2cm} 
Using Theorem \ref{mainthm3}, we now prove Theorem 3, in the following corollary. 
\vspace{0.2cm}    

\vspace{0.2cm} 
\begin{cor}\label{cor4}
We notation as above, we have the following estimate
\begin{align} \label{cor4:eqn}
\sup_{z\in X_{\G}}\big|\bk(z)\big|_{\mathrm{pet}} = O_{\Gamma}\big(k^{5 \slash 2}\big).
\end{align}
\end{cor}  
\begin{proof}
For $k\geq 6$, we have the following estimates for the Gamma function
\begin{align*}
\frac{\Gamma(k\slash 2 - 1/2)}{\Gamma(k\slash2 )}= O\bigg(\frac{1}{\sqrt{k}}\bigg),\,\,\,\frac{\Gamma(k - 3/2)}{\Gamma(k-1 )}= O\bigg(\frac{1}{\sqrt{k}}\bigg).
\end{align*}

The proof of the corollary follows from combining the above estimate with inequality \eqref{mainthm3:eqn} and estimate \eqref{ck2}.
\end{proof}

\vspace{0.1cm}
\begin{rem}\label{rem:covers}
The implied constants in estimates \eqref{eq:theorem2} and \eqref{eq:theorem3} remain stable in covers. Let $\G_{0}\subset \mathrm{SU}\big(n,1),\mathbb{C})$ be a discrete, torsion-free, cocompact subgroup, and let $\G_1\subset \G_{0}$ be a finite index subgroup. Let $X_{\G_1}:\G_{1}\backslash\mathbb{B}^{n}$, $X_{\G_{0}}\backslash\mathbb{B}^{n}$ denote the respective quotient spaces. Then, from the proof of estimate \eqref{eq:theorem2} from Theorem \ref{mainthm2}, for any $z\in\mathbb{B}^n$, we have the following inequality
\begin{align*}
\big|\mathcal{B}_{\G_{1}}^{k}(z)\big|_{\mathrm{pet}}\leq \sum_{\gamma\in\G_{1}}\frac{\ck}{\cosh^{k}\big(\dhyp(z,\gamma z)\slash 2\big)}\leq \sum_{\gamma\in\G_{0}}\frac{\ck}{\cosh^{k}\big(\dhyp(z,\gamma z)\slash 2\big)}=O_{\G_{0}}\big(k^{n}\big).
\end{align*}

Similarly, from the proofs of Theorem \ref{mainthm3} and Corollary \ref{cor4}, we can conclude that the implied constant in estimate  \eqref{eq:theorem3} also remains stable in covers of ball qoutients.  
\end{rem}
\subsection*{Acknowledgements}
 The first and third authors acknowledge the support of INSPIRE research grant DST/INSPIRE/04/2015/002263 and the MATRICS grant MTR/2018/000636. The second author was partially supported by SERB grants EMR/2016/000840 and MTR/2017/000114. 

The first author is very thankful to Dr. Soumya Das, Prof. J\"urg Kramer, and Dr. Antareep Mandal for helpful discussions on  estimates of automorphic forms.

\end{document}